\documentclass{article}

\def\y{\mathbf{y}}

\def\S{\mathcal{S}}

\usepackage[left = 2cm, right = 2cm, top = 1.5cm, bottom = 2.5cm]{geometry}
\usepackage{enumerate}
\usepackage{amsmath,amsthm,amsfonts,amssymb,upgreek,multicol,mathrsfs,pifont,amscd,latexsym,cite,bm,color,fancyhdr}
\usepackage{graphicx}
\usepackage{algorithm,algorithmic,tabularx,url,float,booktabs,tikz}
\usepackage{fancyhdr,lastpage}
\usepackage{multirow}

\newtheorem{theorem}{Theorem}[section]

\newtheorem{lemma}[theorem]{Lemma}
\newtheorem{assumption}[theorem]{Assumption}
\newtheorem{definition}[theorem]{Definition}
\newtheorem{proposition}[theorem]{Proposition}
\newtheorem{remark}[theorem]{Remark}

\newcommand{\mathup}[1]{\text{\textup{#1}}}

\begin{document}

\title{Gradient-based algorithms for multi-objective bi-level optimization\footnote{Citation: Xinmin Yang, Wei Yao, Haian Yin, Shangzhi Zeng, Jin Zhang. Gradient-based algorithms for multi-objective bi-level optimization. SCIENCE CHINA Mathematics, 2024, 67: 1419--1438. Doi: 10.1007/s11425-023-2302-9.}}

\date{}

\author{
	Xinmin Yang$^{1,2}$\quad Wei Yao$^{3,4}$\quad Haian Yin$^{4}$\quad Shangzhi Zeng$^{5}$\quad Jin Zhang$^{4,3,\ast}$\\
	\small $^{1}$National Center for Applied Mathematics in Chongqing, Chongqing {\rm401331}, China\\
	\small $^{2}$School of Mathematical Sciences, Chongqing Normal University, Chongqing {\rm401331}, China\\
	\small $^{3}$National Center for Applied Mathematics Shenzhen, Shenzhen {\rm518000}, China\\
	\small $^{4}$Department of Mathematics, Southern University of Science and Technology, Shenzhen {\rm518055}, China\\
	\small $^{5}$Department of Mathematics and Statistics, University of Victoria, Victoria, British Columbia {\rm V8W 2Y2}, Canada\\
	\small Email: xmyang@cqnu.edu.cn, yaow@sustech.edu.cn, 11930905@mail.sustech.edu.cn,\\ 
	\small zengshangzhi@uvic.ca, zhangj9@sustech.edu.cn
}

\maketitle

\noindent\textbf{Abstract}\quad Multi-Objective Bi-Level Optimization (MOBLO) addresses nested multi-objective optimization problems common in a range of applications. However, its multi-objective and hierarchical bilevel nature makes it notably complex.  Gradient-based MOBLO algorithms have recently grown in popularity, as they effectively solve crucial machine learning problems like meta-learning, neural architecture search, and reinforcement learning.  Unfortunately, these algorithms depend on solving a sequence of approximation subproblems with high accuracy, resulting in adverse time and memory complexity that lowers their numerical efficiency.  To address this issue, we propose a gradient-based algorithm for MOBLO, called gMOBA, which has fewer hyperparameters to tune, making it both simple and efficient.  Additionally, we demonstrate the theoretical validity by accomplishing the desirable Pareto stationarity.  Numerical experiments confirm the practical efficiency of the proposed method and verify the theoretical results.  To accelerate the convergence of gMOBA, we introduce a beneficial L2O neural network (called L2O-gMOBA) implemented as the initialization phase of our gMOBA algorithm. Comparative results of numerical experiments are presented to illustrate the performance of L2O-gMOBA.

\hspace{3pt}

\noindent \textbf{Keywords}\quad multi-objective, bi-level optimization, convergence analysis, Pareto stationary, learning to optimize

\section{Introduction}
\label{intro}

In this paper, we are interested in the Multi-Objective Bi-Level Optimization (MOBLO) problem stated as follows:
\begin{equation}\label{MOBLO}
	\begin{aligned}
		\mathop{\min}\limits_{x \in \mathcal{X},\, y \in \mathbb{R}^{n_2}} 
		& (F_1(x,y) + g_1(x), \ldots, F_m(x,y) + g_m(x)), \\
		\mathrm{s.t.} \quad & 
		y \in \S(x):=\underset{y' \in \mathbb{R}^{n_2}}{\mathrm{argmin}} \, f(x,y'),
	\end{aligned}
\end{equation}
where $F_1(x,y), \dots, F_m(x,y)$ are smooth but possibly nonconvex functions, $g_1(x),\dots,g_m(x)$ are convex but possibly nonsmooth functions, the constraint set $\mathcal{X} \subset \mathbb{R}^{n_1}$ is closed and convex, and $f(x,y)$ is a smooth scalar function and strongly convex with respect to $y$. The vector-valued function $(F_1(x,y) + g_1(x), \ldots, F_m(x,y) + g_m(x))$ given above is called the upper-level multi-objective, while the scalar function $f(x,y)$ is called the lower-level objective. The variables $x\in\mathbb{R}^{n_1}$, $y \in \mathbb{R}^{n_2}$ are called the upper-level (UL) and lower-level (LL) variables, respectively.
MOBLO tackles nested multi-objective optimization structures appearing in applications, where the multiple-objective UL problem must be solved while ensuring the optimality of the LL problem.
The multi-objective and nested nature make MOBLO notoriously challenging, even in the special case where the LL problem is unconstrained and the LL objective is strongly convex with respect to (w.r.t.) the LL variable. 

MOBLO has gained increased attention in recent years as it is highly relevant in practice. 
Most real-world problems require trading off multiple competing objectives, especially in domains such as multi-objective meta-learning \cite{goldblum2020adversarially, ye2021multi}, multi-objective neural architecture search (NAS) \cite{dong2018dpp, tan2019mnasnet, elsken2018efficient, lu2020nsganetv2}, and multi-objective reinforcement learning \cite{vamplew2011empirical, mossalam2016multi, yang2019generalized, abdolmaleki2020distributional}. In these domains, a learner is required to discover a model that performs well across different objectives, such as prediction quality, efficiency, fairness, or robustness. MOBLO allows us to optimize trade-offs between different performance measures and yield better solutions than the single-objective bi-level optimization (BLO) problems by taking a multi-objective perspective. Recent multi-objective approaches have proposed fixing Generative Adversarial Networks (GANs) instability issues with multiple discriminators \cite{neyshabur2017stabilizing, albuquerque2019multi}.

MOBLO's challenges stem from the need to determine a solution that balances different performance metrics while solving optimization problems at diverse levels with distinct objectives. For example, conventional meta-learning methods and applications assume that the contribution of each task or instance to the meta-learner is equal. Therefore, the meta-training step can be formulated as a BLO problem \cite{TimothyMHospedales2020MetaLearningIN, liu2021investigating}. However, when the test tasks come from a different distribution than the training tasks distribution, the existing meta-learning techniques often fail to generalize well \cite{chen2018closer}. Model-Agnostic Meta-Learning (MAML) \cite{finn2017model} also often fails to address domain shift between base and novel classes in few-shot learning since it assumes equal weights to all samples and tasks during meta-training \cite{killamsetty2022nested}. 

To avoid a priori trade-offs, meta-learning with multiple objectives can be formulated as a MOBLO problem, as tackled in some prior works. However, existing works utilize multi-objective bi-level evolutionary algorithms, which may fail in solving large-scale MOBLO problems due to the lack of gradient information and higher computational complexity. To address this issue, \cite{ye2021multi} proposed a gradient-based MOBLO algorithm with convergence guarantee by solving the lower-level and upper-level subproblems alternatively via the gradient descent method and the gradient-based multi-objective optimization method, respectively. 
However, this algorithm requires the LL problem trajectory length to go to infinity and find out the Pareto optimal solution of each approximation multi-object problem, which is difficult to implement in practice.
Hence, a practical issue lingers:

{\it Can we design a new gradient-based MOBLO algorithm that is easy to implement and provably converges?}

\subsection{Main Contributions}

In response to the above issues, 
this paper introduces a simple yet highly efficient gradient-based algorithm for MOBLO, namely gMOBA, with a convergence guarantee. 
Numerical experiments confirm the practical efficiency of the proposed method and verify the theoretical convergence results.
To further accelerate the convergence of gMOBA, motivated by the notable performance of model-based L2O methods, 
we propose a useful L2O neural network by unrolling and truncating our gMOBA algorithm. 
The main contributions of this paper are as follows:
\begin{itemize}
	\item We propose a new gradient-based algorithm for MOBLO, namely gMOBA. It is single-loop and Hessian inverse-free with a convergence guarantee. Unlike the existing MOBLO methods that sequentially update the upper-level variable after fully updating the lower-level variable, our algorithm avoids solving a series of time-consuming subproblems, making it more powerful.
	
	\item Using a useful nonsmooth Lyapunov function approach, we justify the convergence towards the desirable Pareto stationarity of the proposed algorithm. These results are new, even in the context of single-objective bi-level optimization.
	
	\item To accelerate the convergence of gMOBA, motivated by the recent promising performance of L2O, we introduce a beneficial L2O neural network (called L2O-gMOBA) implemented as the initialization phase of our gMOBA algorithm. Comparative results of numerical experiments are presented to illustrate the efficiency of L2O-gMOBA.
	
\end{itemize}

\subsection{Related Work}

{\bf Multi-Objective Optimization (MOO).} 
MOO, also known as Pareto Optimization, aims to optimize a set of potentially conflicting objectives simultaneously. 
It has grown increasingly popular due to its broad applications in machine learning, particularly in multi-task learning \cite{ruder2017overview, sener2018multi, lin2019pareto, yu2020gradient, mahapatra2020multi, liu2021conflict, momma2022multi}, multi-objective reinforcement learning (RL) \cite{van2014multi, chen2019meta, abdolmaleki2020distributional}, federated learning \cite{mohri2019agnostic, hu2022federated}, and so on. 
A standard approach for MOO is the weighted (scalarized) method, which minimizes sums of the different objectives for various weight combinations. This technique is straightforward and commonly used in machine learning, mainly because conventional learning algorithms are only capable of handling scalar cost functions. However, the weighted approach could be extremely inefficient \cite{jin2008pareto}. In recent years, the multiple gradient descent algorithm (MGDA) \cite{fliege2000steepest, desideri2012multiple} has gained popularity in machine learning. It generates a gradient vector for discovering Pareto solutions. For instance, MGDA has been applied to multi-task learning in \cite{sener2018multi} and to federated learning in \cite{hu2022federated}. Latterly, for nonsmooth MOO problems, MGDA has been improved by a proximal point method in \cite{bonnel2005proximal}, a subgradient method in \cite{da2013subgradient},  proximal gradient methods in \cite{tanabe2019proximal}, 
and a Barzilai-Borwein descent method in \cite{chen2023barzilai} recently. 
MOBLO, on the other hand, addresses nested multi-objective optimization structures present in real-world applications. Its hierarchical bilevel structure introduces additional difficulty, as computing the multi-hypergradient, i.e., the vector of multiple gradients of UL total objectives, is prohibitively expensive, requiring solving a series of approximation subproblems with high accuracy \cite{ye2021multi}.

{\bf Bi-level Optimization (BLO).} 
BLO addresses nested optimization structures present in real-world applications. 
In the last decade, it has received increasing attention, particularly in deep learning, such as hyperparameter optimization (\cite{pedregosa2016hyperparameter, franceschi2017forward, mackay2018self}), meta-learning (\cite{franceschi2018bilevel, zugner2018adversarial, rajeswaran2019meta, KaiyiJi2020ConvergenceOM}), and neural architecture search (\cite{liu2018darts, liang2019darts+, chen2019progressive}), among other areas. 
A variety of gradient-based BLO algorithms have gained popularity because of their effectiveness and simplicity. Most of them rely on various hypergradient approximations, i.e., approximation of the gradient of the UL objective. For example, an approximate hypergradient can be straightforwardly calculated by automatic differentiation based on the optimization trajectory of the LL variable in the iterative differentiation (ITD) based approach \cite{maclaurin2015gradient, franceschi2017forward, shaban2019truncated, grazzi2020iteration, liu2020generic, ji2021bilevel}. Another approach, namely the approximate implicit differentiation (AID) based method \cite{pedregosa2016hyperparameter, ghadimi2018approximation, rajeswaran2019meta, ji2022will}, exploit implicit differentiation to derive an analytical expression of the hypergradient and then estimates the Hessian-inverse-vector product by solving a linear system accurately. 
The novel BLO frameworks presented in \cite{dagreou2022framework} and \cite{liu2023bilevel} are particularly relevant to our work. They allow for the simultaneous evolution of UL, LL variables, and the solution of the linear system, enabling stochastic and global variance reduction algorithms. However, they focus on smooth BLO, while we generalize this to MOBLO with nonsmooth multiple UL objectives.

{\bf Learning to Optimize (L2O).}
L2O is an emerging approach that leverages machine learning to develop an optimization method by training, i.e., learning from its performance on sample problems \cite{chen2022learning}. There are two mainstream L2O approaches: model-free and model-based. A model-free L2O approach is generally to  learn a parameterized update rule of optimization without taking the form of any analytic update \cite{andrychowicz2016learning}. 
In contrast, model-based L2O methods model their iterative update rules through a learnable architecture inspired by analytic optimization algorithms.
Most model-based L2O methods take one of the two following mainstream approaches. The first approach is known as plug and play (PnP), whose key idea is to plug a pre-trained neural network into part of the update for an optimization algorithm, and then play by immediately applying the modified algorithm to problems sampled from the same task distribution \cite{venkatakrishnan2013plug}. The second approach is known as algorithm unrolling, which unrolls and truncates optimization algorithm into the structure of a neural network \cite{monga2021algorithm}. 
The parameters in unrolled schemes are trained end-to-end using the final iterate as a function of each learnable weights whereas training occurs separately for PnP. 
The typical algorithm unrolling methods
include variations of the iterative shrinkage thresholding algorithms (ISTA) and the alternating direction method of multipliers (ADMM), see, e.g., \cite{gregor2010learning,sprechmann2013supervised,sun2016deep,liu2019alista}. 

\section{Preliminaries}
Throughout this paper, we define the relation $\le (<)$ in $\mathbb{R}^m$ as $a \le b \ (a<b)$
if and only if $a_i \le b_i \ (a_i < b_i) $ for all $i = 1,\ldots, m$.  
The following standing assumptions on the UL multi-objective and the LL objective are adopted throughout this paper. 
\begin{assumption}\label{Assump}  \phantom{1}
	\begin{enumerate}[(a)]
		\item For any $x \in \mathcal{X}$, the LL objective $f(x, y)$ is $\mu$-strongly convex with respect to the LL variable $y$.
		
		\item The LL objective $f(x,y)$ is twice continuously differentiable and with Lipschitz continuous first and second order derivatives on $\mathcal{X} \times \mathbb{R}^{n_2}$.
		
		\item For $i = 1, \ldots, m$, the smooth part $F_i(x,y)$ of the UL objective is continuously differentiable and bounded below on $\mathcal{X} \times \mathbb{R}^{n_2}$, its first order derivative is Lipschitz continuous and bounded on $\mathcal{X} \times \mathbb{R}^{n_2}$.
		
		\item Each nonsmooth part $g_i(x)$ of the UL multi-objective is convex and continuous on an open set $\mathcal{O} \supseteq \mathcal{X}$.
	\end{enumerate}
\end{assumption}

The above assumptions are standard in the BLO literature, see, e.g.,  \cite{ghadimi2018approximation, ji2021bilevel, khanduri2021near}.
Assumption \ref{Assump}(a) leads to the uniqueness of $\S(x)$ for any $x \in \mathcal{X}$. And then Problem \eqref{MOBLO} has the following equivalent single-level multi-objective optimization reformulation,
\begin{equation}\label{multobj}
	\min_{x } \Phi(x) := (\varphi_1(x) + g_1(x), \ldots, \varphi_m(x)+ g_m(x)),
\end{equation}
where $y^*(x)= {\mathrm{argmin}}_{y}  f(x,y)$ and $\varphi_i(x) = F_i(x, y^*(x)) $.

We recall the definition of the Pareto optimal solution of multi-objective optimization as follows.
\begin{definition}
	A point $\bar{x} \in \mathbb{R}^{n_1}$ is called {\it Pareto optimal} for $\Phi$, if there is no $x \in \mathbb{R}^{n_1}$ such that $\Phi(x) \le \Phi(\bar{x})$ and $\Phi(x) \neq \Phi(\bar{x})$. Similarly, $\bar{x} \in \mathbb{R}^{n_1}$ is called {\it weakly Pareto optimal} for $\Phi$, if there is no $x \in \mathbb{R}^{n_1}$ such that  $\Phi(x) < \Phi(\bar{x})$. 
\end{definition}

It is known that Pareto optimal points are always weakly Pareto optimal, and the converse is not always true. 
\begin{definition}
	We say that $\bar{x} \in \mathbb{R}^{n_1}$ is {\it Pareto stationary }for $\Phi$ if
	\[
	\max_{i = 1,\ldots,m} \Phi_i'(x;d) \ge 0 \quad \forall\, d \in \mathbb{R}^{n_1},
	\]
	where $\Phi_i'(x;d)$ is the directional derivative of $\Phi_i$ at $x$ in the direction $d$.
\end{definition}
With this definition, as shown in Lemma 2.2 of \cite{tanabe2019proximal}, weakly Pareto optimal points are always Pareto stationary. 
\begin{lemma} \cite{tanabe2019proximal}
	If $\bar{x} \in \mathbb{R}^{n_1}$ is {\it weakly Pareto optimal} for $\Phi$, then $\bar{x} \in \mathbb{R}^{n_1}$ is {\it Pareto stationary }for $\Phi$.
\end{lemma}

But the converse is not always true. If every component $\Phi_i$ is convex, then Pareto stationarity implies weak Pareto optimality. Furthermore, if every component $\Phi_i$ is strictly convex, then Pareto stationary points are also Pareto optimal \cite{tanabe2019proximal}.

\section{Gradient-based Multi-Objective Bilevel Algorithm}

In this section, we will present and conduct a convergence analysis of our main algorithm, namely the gradient-based Multi-Objective Bilevel Algorithm (gMOBA), designed to solve MOBLO in \eqref{MOBLO}. 

\subsection{Algorithm Formulation}

As discussed in the previous section, under Assumption \ref{Assump}(a),
a unique solution can be derived for the LL problem, represented as $y^*(x)$, when $x$ is fixed.
Therefore, the LL solution can be incorporated into the smooth part of the UL multi-objective as $\varphi(x):=(F_1(x, y^*(x)),\dots, F_m(x, y^*(x)))$, and Problem \eqref{MOBLO} is equivalently reformulated into the following single-level multi-objective optimization problem on the variable $x$,
\begin{equation*}
	\min_{x \in \mathcal{X}} \Phi(x) := (\varphi_1(x) + g_1(x), \ldots, \varphi_m(x)+ g_m(x)).
\end{equation*}
In particular,
by utilizing the chain rule and implicit function theorem, we can show that for each $i=1,\dots,m$, $\varphi_i(x)$ is differentiable and 
\begin{equation}\label{hypergradient}
	\nabla\varphi_i(x)=\nabla_{x} F_i(x,y^*(x))-\nabla_{xy}^2 f(x,y^*(x)) v_i^*(x),
\end{equation}
where $v_i^*(x)$ is the solution of the following linear system
\begin{equation}\label{eqv}
	\nabla_{yy}^2 f (x,y^*(x)) v_i^*(x)=\nabla_{y} F_i (x,y^*(x)).
\end{equation}
However, the computation of $\nabla \varphi_i$ is challenging and expensive due to the nested nature of function $\varphi_i(x)$. 
With given $x$, in order to calculate $\nabla\varphi_i(x)$, we need to solve the LL problem $\min_y f(x,y)$ for $y^*(x)$ and the linear system \eqref{eqv} for $v^*_i(x)$, which is costly. To address this issue, we propose to use an easy-to-compute approximation of $\nabla \varphi_i$ in the developing of gMOBA.
Inspired by the single-loop gradient-based algorithm proposed in  \cite{dagreou2022framework} for solving single objective BLO problems, we consider the following computation process for generating approximation $d_{\varphi_i}^k$ of $\nabla\varphi_i(x^k)$ for $i = 1, \ldots, m$ at each iteration as
\begin{equation}
	d_{\varphi_i}^k  =\nabla_{x} F_i(x^k,y^k) - \nabla_{xy}^2 f(x^k,y^k)v_i^k, 
\end{equation}
with $(y^k, v^k_i)$ being approximation to $(y^*(x^k), v^*_i(x^k))$ updated at each iteration as
\begin{align}
	y^{k+1} &= y^k - \beta \nabla_{y} f(x^k,y^k), \label{yupdate}\\
	v^{k+1}_i &= v^k_i - \eta_k d_{v_i}^k, \label{vupdate}
\end{align}
where $\beta$ and $\eta_k$ are positive step sizes, and
\begin{equation}\label{directionv}
	d_{v_i}^k = \nabla_{yy}^2 f (x^k,y^k) v_i^k-\nabla_{y} F_i (x^k,y^k).
\end{equation}
Note that \eqref{yupdate} is exactly doing a gradient descent step on $f(x^k, \cdot)$ from $y^k$, following the direction $-\nabla_{y} f(x^k,y^k)$ to  approximate $y^*(x^k)$. And \eqref{vupdate} corresponds to solve $\nabla_{yy}^2 f (x^k,y^k) v_i=\nabla_{y} F_i (x^k,y^k)$ by following the direction $-d_{v_i}^k$ from $v_i^k$ to approximate $v^*_i(x^k)$.

With the approximation $d_{\varphi_i}^k$ of $\nabla\varphi_i(x^k)$ for $i = 1, \ldots, m$ at each iteration on $(x^k, y^k, v^k_i)$, inspired by the proximal gradient methods for multi-objective optimization in  \cite{tanabe2019proximal}, we propose solving the following strongly convex optimization problem to update the next $x^{k+1}$,
\begin{equation}\label{updatex}
	x^{k+1}=  \underset{x \in \mathcal{X}}{\mathrm{argmin}}  \left\{\max_{i = 1, \ldots, m} h_i ^k (x) + \frac{1}{2\alpha_k} \|x - x^k\|^2 \right\},
\end{equation}
where 
\begin{equation}\label{hdefine}
	h_i ^k (x) =  \langle d_{\varphi_i}^k, x - x^{k}\rangle  + g_i(x) - g_i(x^k).
\end{equation}

Now, we are ready to present our proposed gradient-based algorithm for solving MOBLO \eqref{MOBLO}, namely gMOBA, in Algorithm \ref{alg1}. It should be noticed that the proposed gMOBA is a parallelizable algorithm because the updates of $(x^k, y^k, v^k)$ can be implemented simultaneously.

\begin{algorithm}[thb]
	\caption{gradient-based Multi-Objective Bilevel Algorithm (gMOBA) }
	\label{alg1}\small
	\begin{algorithmic}[1] 
		\STATE {\bfseries Input:} initial points $x^0,y^0,v^0$, positive stepsizes $\alpha_k,\beta,\eta_k$;
		\FOR{$k=0,1,\dots,K-1$} 
		\STATE update $y^{k+1}=y^k-\beta\nabla_{y}f(x^k,y^k)$;
		
		\STATE for $i = 1, \ldots, m$, update $v^{k+1}_i = v^k_i - \eta_k d_{v_i}^k$ with $d_{v_i}^k$ defined in Eq. \eqref{directionv};
		
		\STATE  update $x^{k+1}$ by solving
		\[
		x^{k+1}=  \underset{x \in \mathcal{X}}{\mathrm{argmin}}  \left\{\max_{i = 1, \ldots, m} h_i ^k (x) + \frac{1}{2\alpha_k} \|x - x^k\|^2 \right\}, 
		\]
		where $h_i ^k (x)$ is defined in Eq. \eqref{hdefine}.
		\ENDFOR
	\end{algorithmic}
\end{algorithm}

\begin{remark}
	It should be mentioned that though our proposed gMOBA is designed for the MOBLO \eqref{MOBLO}, where the LL objective is a scalar function, gMOBA can also handle a special class of MOBLO with multiple objectives on both the upper- and lower-levels. 
	We consider a MOBLO problem in the form as follows,
	\begin{equation}\label{blo_problem0}
		\begin{aligned}
			\mathop{\min}\limits_{x \in \mathcal{X},\, y \in \mathbb{R}^{n_2} } 
			& (F_1(x,y) + g_1(x), \ldots, F_m(x,y) + g_m(x)), \\
			\mathrm{s.t.} \quad & 
			y \in \mathcal{P}(x),
		\end{aligned}
	\end{equation}
	where $\mathcal{P}(x)$ is the set of efficient solutions of the lower-level  multi-objective optimization problem:
	\begin{equation}
		\min_{y\in\mathbb{R}^{n_2}}\, \left(f_1(x, y), \dots, f_l(x, y)\right).
	\end{equation}
	This problem is also known as the semivectorial bilevel optimization problem (cf. \cite{bonnel2006semivectorial, dempe2013new, andreani2019bilevel}).
	We suppose that $f_1, \dots, f_l$ are all strongly convex w.r.t. $y$. It holds that $y$ in $\mathcal{P}(x)$ if and only if there is a $\lambda\in\Delta_l:=\{\lambda\in\mathbb{R}^l: \lambda\geq0, e^T \lambda=1\}$ such that $y$ solves $\min_{y}\sum_{j=1}^l \lambda_j f_j(x,y)$, see Section 3.1 in \cite{ehrgott2005multicriteria}. Then this classical result in MOO implies that Problem \eqref{blo_problem0} can be equivalently reformulated as 
	\begin{equation}\label{blo_problem3}
		\begin{aligned}
			\mathop{\min}\limits_{(x, \lambda) \in \mathcal{X}\times \Delta_l,\, y \in\mathbb{R}^{n_2}} 
			& (F_1(x,y) + g_1(x), \ldots, F_m(x,y) + g_m(x)),  \\
			\mathrm{s.t.} \quad & 
			y \in \S(x, \lambda):=\underset{y \in \mathbb{R}^{n_2}}{\mathrm{argmin}} \,\sum_{j=1}^l \lambda_j f_j(x,y),
		\end{aligned}
	\end{equation}
	which is a special case of the MOBLO \eqref{MOBLO}. 
	Applying the weighted-sum-scalarization technique to the multiobjective lower-level problem of \eqref{blo_problem0} is not new, cf. \cite{dempe2019semivectorial}. 
	It can be shown that if \eqref{blo_problem0} satisfies Assumption \ref{Assump}, and Assumption \ref{Assump}(a) is replaced by that for any $x \in \mathcal{X}$, the LL objectives $f_1, \dots, f_l$ are all $\mu$-strongly convex with respect to the LL variable $y$, then \eqref{blo_problem3} satisfies Assumption \ref{Assump}. And then, under these assumptions, our proposed gMOBA can be used for solving \eqref{blo_problem0}  by applying Algorithm \ref{alg1} on \eqref{blo_problem3}.
\end{remark}

 \subsection{Convergence Analysis}
\subsubsection{Preliminary Results on gMOBA}
Before presenting the convergence analysis on gMOBA, we first recall the following useful lemmas. We recall that Assumptions \ref{Assump} is assumed to hold throughout this part and let  $C_{F_{\y}}$ denote the upper bound of $\|\nabla_{y}F(x,y)\|$ on $\mathcal{X} \times \mathbb{R}^{n_2}$, and $L_{F_{y}}$, $L_{f_{y}}$, $L_{f_{yy}}$ and $L_{f_{xy}}$ denote the Lipschitz constant of $\nabla_{y}F(x,y)$, $\nabla_{y}f(x,y)$, $\nabla_{yy}^2f(x,y)$ and $\nabla_{xy}^2f(x,y)$ on $\mathcal{X} \times \mathbb{R}^{n_2}$, respectively.

\begin{lemma}\label{lemma*}\cite{ghadimi2018approximation}
	The following statements hold.
	\begin{enumerate}[(i)]
		\item The function $y^*(x)$ is Lipschitz continuous in $x \in \mathcal{X}$, i.e., for all $x, x' \in \mathcal{X}$, 
		\begin{equation}
			\left\| y^*(x)- y^*(x')\right\|\leq \frac{L_{f_{y}}}{\mu}
			\|x-x'\|.
		\end{equation}
		
		\item Both $\nabla y^*(x)$ and $v^*(x)$ are uniformly bounded for $x \in \mathcal{X}$, i.e., 
		\begin{equation}
			\left\|\nabla y^*(x)\right\|=\left\|\left[\nabla_{yy}^2 f(x,y^*(x))\right]^{-1}\nabla_{yx}^2 f(x,y^*(x))\right\|
			\leq \frac{L_{f_{y}}}{\mu},
		\end{equation}
		and
		\begin{equation}
			\|v_i^*(x)\| = \left\|\left[\nabla_{yy}^2 f(x,y^*(x))\right]^{-1}\nabla_{y} F_i(x,y^*(x))\right\|\leq \frac{C_{F_{\y}}}{\mu}, \quad i = 1, \ldots, m,
		\end{equation}
		where $v_i^*(x) := \left[\nabla_{yy}^2 f(x,y^*(x))\right]^{-1}\nabla_{y} F_i(x,y^*(x))$, $i = 1, \ldots, m$.
		\item The functions $v_i^*(x)$, $i = 1, \ldots, m,$ are Lipschitz continuous in $x \in \mathcal{X}$, that is, for all $x, x' \in \mathcal{X}$,
		\begin{align}
			\left\| v_i^*(x)- v_i^*(x')\right\|\leq \frac{L_{v}}{\mu^3}
			\|x-x'\|, \quad i = 1, \ldots, m,
		\end{align}
		where 
		\begin{align*}
			L_{v}=C_{F_{y}} L_{f_{yy}} L_{f_{y}}+\mu(C_{F_{y}} L_{f_{yy}}+L_{F_{y}} L_{f_{y}})+\mu^2 L_{F_{y}}.
		\end{align*}
	\end{enumerate}
\end{lemma}

\begin{lemma}\label{lem12}\cite{ghadimi2018approximation}
	There exists $L_{\varphi} > 0$ such that for each $i = 1, \ldots, m$,
	\begin{equation}
		\|\nabla \varphi_i(x) -\nabla \varphi_i(x')\|\leq \frac{L_\varphi}{\mu^3}
		\|x-x'\|\quad \forall \,x, x'\in \mathcal{X}.
	\end{equation}
\end{lemma}

\subsubsection{Fundamental Descent Lemmas}

To analyze the convergence of gMOBA towards Pareto stationary points, we identify an intrinsic Lyapunov (potential) function for the generated sequence $\{(x^k,y^k,v^k_i)\}$ by gMOBA as: for $i = 1, \ldots, m$,
\begin{equation}\label{Lyapunov}
	\begin{aligned}
		V_i^k:= &\varphi_i(x^k) + g_i(x^k)
		+\|y^k-y^*(x^k)\|^2
		+ \|v_i^k-v^*_i(x^k)\|^2.
	\end{aligned}
\end{equation}
Note that the first term quantifies the $i$-th overall UL objective, the second one delineates the LL solution error, and the last term characterizes the error of $v_i$ solving the linear system \eqref{eqv}. Note that the Lyapunov function defined here is nonsmooth, different from the existing results in BLO literature. 

Now we study the descent of these  Lyapunov functions, from which one can get the convergence of $
\|x^{k+1} - x^{k}\|$, $\|y^{k}- y^*(x^{k}) \|$ and $ \|v_i^{k} - v_i^*(x^{k})\| $, $ i = 1, \ldots, m$. We first upper-bound the descent of the multiple nonsmooth overall UL objectives.

\begin{lemma}\label{lemmaFk}
	The sequence of $(x^k,y^k,v^k_i)$ generated by Algorithm\ref{alg1} satisfies that for each $i = 1, \ldots, m$,
	\begin{equation}
		\begin{aligned}
			\Phi_i(x^{k+1}) - \Phi_{i}(x^k)
			\leq &-\frac{1}{2}\left(\frac{1}{2\alpha_k}-\frac{L_\varphi}{\mu^3}\right)\|x^{k+1}-x^k\|^2
			+ 2\alpha_k\left( L_{F_{x}} + \frac{C_{F_{\y}} L_{f_{xy}}}{\mu}\right)^2 \|y^{k}-y^*(x^k)\|^2 \\
			&+ 2\alpha_k L_{f_{y}}^2 \|v_i^{k}-v^*_i(x^k)\|^2.
		\end{aligned}
	\end{equation}
\end{lemma}

\begin{proof}
	By the update rule of $x^{k+1}$, we have for each $i = 1, \ldots, m$,
	\begin{equation}\label{lem3_eq1}
		\begin{aligned}
			&\langle d_{\varphi_i} ^k, x^{k+1} - x^{k}\rangle + g_i(x^{k+1}) - g_i(x^k) + \frac{1}{2\alpha_k}\|x^{k+1} - x^{k}\|^2 \\\le \, & \max_{i = 1, \ldots, m}\left\{ \langle d_{\varphi_i}^k, x^{k+1} - x^{k}\rangle + g_i(x^{k+1}) - g_i(x^k) \right\} + \frac{1}{2\alpha_k}\|x^{k+1} - x^{k}\|^2 \\
			\le\, &\max_{i = 1, \ldots, m}\left\{ \langle d_{\varphi_i}^k, x^{k} - x^{k}\rangle + g_i(x^{k}) - g_i(x^k) \right\} + \frac{1}{2\alpha_k}\|x^{k} - x^{k}\|^2 = 0.
		\end{aligned}
	\end{equation}
	Next, by the $\frac{L_\varphi}{\mu^3}$-smooth of $\varphi_{i}(\cdot)$ established in Lemma \ref{lem12}, we have that for each $i = 1, \ldots, m$,
	\begin{equation}\label{lem3_eq2}
		\varphi_i(x^{k+1}) - \varphi_i(x^{k}) \leq \big\langle \nabla \varphi_{i}(x^k), x^{k+1} - x^k \big\rangle+\frac{L_\varphi}{2\mu^3}\|x^{k+1}-x^k\|^2 .
	\end{equation}
	Summing up \eqref{lem3_eq1} and \eqref{lem3_eq2}, we have that for each $i = 1, \ldots, m$,
	\[
	\begin{aligned}
		& \varphi_i(x^{k+1})  + g_i(x^{k+1}) - \varphi_i(x^{k}) - g_i(x^k) + \frac{1}{2\alpha_k}\| x^{k+1} - x^{k}\|^2 \\  
		\leq\ & \big\langle \nabla \varphi_{i}(x^k) - d_{\varphi_i}^k, x^{k+1} - x^k \big\rangle + \frac{L_\varphi}{2\mu^3}\|x^{k+1}-x^k\|^2 .
	\end{aligned}
	\]
	Note that
	\[
	\begin{aligned}
		\big\langle \nabla \varphi_{i}(x^k) -d_{\varphi_i}^k, x^{k+1}-x^k \big\rangle 
		& \le \alpha_k \| \nabla \varphi_{i}(x^k) - d_{\varphi_i}^k\|^2 + \frac{1}{4\alpha_k}\|x^{k+1}-x^k\|^2.
	\end{aligned}
	\]
	Since 
	\begin{equation*}
		\nabla \varphi_{i}(x^k)=\nabla_{x} F_i(x^k,y^*(x^k))-\nabla_{xy}^2f(x^k,y^*(x^k)) v_i^*(x^k),
	\end{equation*}
	we have
	\[
	\begin{aligned}
		\| 	\nabla \varphi_{i}(x^k) - d_{\varphi_i}^k\| = &\|\nabla \varphi_{i}(x^k) - \nabla_{x} F_i(x^k,y^{k}) + \nabla_{xy}^2 f(x^k,y^{k})v_i^{k}\|\\
		\leq  & \|\nabla_{x} F_i(x^k, y^*(x^k))-\nabla_{x} F_i(x^k,y^{k})\|
		+\left\|\left[\nabla_{xy}^2 f(x^k, y^*(x^k))-\nabla_{xy}^2 f(x^k,y^{k})\right]  v^*_i(x^k)\right\|\\
		&+\left\|\nabla_{xy}^2 f(x^k, y^{k})[v^*_i(x^k)-v_i^{k}]\right\|\\
		\leq &\left( L_{F_{x}} + \frac{C_{F_{\y}} L_{f_{xy}}}{\mu}\right) \|y^{k}-y^*(x^k)\|
		+L_{f_{y}} \|v_i^{k}-v^*_i(x^k)\|.
	\end{aligned}
	\]
	The desired result follows from the inequality $\bigl(\sum_{i=1}^{r} a_i\bigr)^2\leq r\sum_{i=1}^{r} a_i^2$. 
	
\end{proof}

Note that the descent of the overall UL objectives depends on the errors of $y^k$ and $v_i^k$. We next analyze these errors. 

\begin{lemma}\label{lemmayv}
	If we choose 
	\begin{equation}
		\beta\leq\frac{2}{\mu+L_{f_y}},\quad \eta_k\leq \frac{1}{L_{f_{y}}}.
	\end{equation}
	Then the sequence of $(x^k,y^k,v^k_i)$ generated by gMOBA satisfies
	\begin{equation*}
		\|y^{k+1}-y^*(x^k)\|^2
		\leq 
		\left(1-{\mu }\beta\right)\|y^k-y^*(x^k)\|^2,
	\end{equation*}
	and for each $i = 1, \ldots, m$,
	\begin{align*}
		\|v_i^{k+1}-v^*_i(x^k)\|^2
		\leq &
		\left(1-\eta_k \mu\right)\|v^k_i-v^*_{i}(x^k)\|^2 + 2\eta_k \frac{\left(L_{f_{yy}}C_{F_y} +L_{F_y}\mu\right)^2}{\mu^3}\|y^{k}-y^*(x^k)\|^2.
	\end{align*}
\end{lemma}
\begin{proof}
	The first estimation follows from Theorem 10.29 in  \cite{beck2017first}.
	By the update of $v_i^{k+1}$ and using $\nabla_{yy}^2 f(x^k,y^*(x^k))v^*_i(x^k)= \nabla_{y} F(x^k,y^*(x^k))$, we have 
	\begin{align*}
		v_i^{k+1} - v^*_i(x^k)
		=\,&v_i^k-v^*_i(x^k)-\eta_k\left(\nabla_{yy}^2 f(x^k,y^{k})v_i^k - \nabla_{y} F(x^k,y^{k}) \right) \\
		=\,&\left[ I - \eta_{k} \nabla_{yy}^2 f(x^k,y^{k})\right](v_i^k - v^*_i(x^k))
		-\eta_{k}\left[\nabla_{yy}^2 f(x^k,y^{k}) - \nabla_{yy}^2 f(x^k,y^*(x^k)) \right]v^*_i(x^k) \\
		& - \eta_k \left(\nabla_{y} F(x^k,y^*(x^k)) -  \nabla_{y} F(x^k,y^{k})\right).
	\end{align*}
	Thus, for any $\varepsilon > 0$, we have 
	\begin{equation}\label{lem4_eq1}
		\begin{aligned}
			\| v_i^{k+1} - v^*_i(x^k) \|^2 \le \, &(1+ \varepsilon) \| \left[ I - \eta_{k} \nabla_{yy}^2 f(x^k,y^{k})\right](v_i^k - v^*_i(x^k))  \|^2 \\
			& + \Big(1 + \frac{1}{\varepsilon}\Big) \eta_k^2\| \left[\nabla_{yy}^2 f(x^k,y^{k}) - \nabla_{yy}^2 f(x^k,y^*(x^k)) \right] v^*_i(x^k) \\
			& + \| \nabla_{y} F(x^k,y^*(x^k)) -  \nabla_{y} F(x^k,y^{k}) \|^2.
		\end{aligned}
	\end{equation}
	As $f(x,y)$ is $\mu$-strongly convex with respect to $y$, we have $\nabla_{yy}^2 f(x^k,y^{k}) \succeq \mu I$ and when $\eta_k\leq \frac{1}{L_{f_{y}}}$. it holds that
	\[
	\| \left[ I - \eta_{k} \nabla_{yy}^2 f(x^k,y^{k})\right](v_i^k - v^*_i(x^k))  \| \le (1- \eta_{k}\mu)\| v_i^k - v^*_i(x^k) \|.
	\]
	Furthermore, the Lipschitz continuity of $\nabla_{yy}^2 f$ and $\nabla_{y} F$ yields that
	\[
	\begin{aligned}
		&\left\| \left[\nabla_{yy}^2 f(x^k,y^{k}) - \nabla_{yy}^2 f(x^k,y^*(x^k)) \right] v^*_i(x^k) + \nabla_{y} F(x^k,y^*(x^k)) -  \nabla_{y} F(x^k,y^{k}) \right\| \\
		\le & (L_{f_{yy}}\|v^*_i(x^k)\| + L_{F_y} ) \|y^{k} - y^*(x^k)\| \\
		\le & \frac{L_{f_{yy}}C_{F_y} +L_{F_y}\mu}{\mu} \|y^{k}-y^*(x^k)\|,
	\end{aligned}
	\]
	where the last equality follows from Lemma \ref{lemma*}.
	Then, taking $\varepsilon = \eta_k\mu$ in \eqref{lem4_eq1} implies that
	\[
	\begin{aligned}
		\| v_i^{k+1} - v^*_i(x^k) \|^2 \le (1+ \eta_{k}\mu)\left(1-\eta_k \mu\right)^2\|v^k_i-v^*_{i}(x^k)\|^2 + \left(1 + \frac{1}{\eta_k \mu}\right)\eta_k^2\left(\frac{L_{f_{yy}}C_{F_y} +L_{F_y}\mu}{\mu} \right)^2 \|y^{k}-y^*(x^k)\|^2.
	\end{aligned}
	\]
	Then the conclusion follows from the fact that $\eta_k^2 \le \eta_k/L_{f_{y}} \le \eta/\mu$.
\end{proof}

\begin{lemma}\label{lemmayv2}
	If we choose 
	\begin{equation}
		\beta\leq\frac{2}{\mu+L_{f_y}},\quad \eta_k\leq \frac{1}{L_{f_{y}}}.
	\end{equation}
	Then the sequence of $(x^k,y^k,v^k_i)$ generated by gMOBA satisfies
	\begin{align*}
		\|y^{k+1}-y^*(x^{k+1})\|^2 
		\leq &  
		\left(1- \frac{1}{2}{\mu }\beta\right) \|y^{k}- y^*(x^{k}) \|^2 
		+ \left(1+ \frac{2}{\mu \beta} \right) \frac{L_{f_{y}}^2}{\mu^2}\|x^{k+1}-x^{k}\|^2.
	\end{align*}
	and for each $i = 1, \ldots, m$,
	\begin{align*}
		\|v_i^{k+1}-v^*_i(x^{k+1})\|^2 
		\le 
		& \left( 1 + \frac{2}{\mu \eta_k} \right)\frac{L_{v}^2}{\mu^6}	\|x^{k+1}-x^{k}\|^2
		+ 2\eta_k \left(1 + \frac{1}{2}\mu \eta_k\right)  \frac{\left(L_{f_{yy}}C_{F_y} +L_{F_y}\mu\right)^2}{\mu^3}\|y^{k}-y^*(x^k)\|^2 \\
		& + \left(1- \frac{1}{2} \mu \eta_k \right) \|v_i^{k}- v_i^*(x^{k}) \|^2.
	\end{align*}
\end{lemma}
\begin{proof}
	We have 
	\[
	\begin{aligned}
		\quad ~ \|y^{k+1}-y^*(x^{k+1})\|^2 \le (1+ \varepsilon) \|y^{k+1}- y^*(x^{k}) \|^2 + (1+ \frac{1}{\varepsilon }) \|y^*(x^{k+1}) -y^*(x^k)\|^2.
	\end{aligned}
	\]
	Taking $\varepsilon = \frac{1}{2}\mu \beta$, then Lemma  \ref{lemmayv}  and $		\|y^*(x^k) - y^*(x^{k-1})\|^2 \le  \frac{L_{f_{y}}^2}{\mu^2}\|x^k-x^{k-1}\|^2$ from Lemma \ref{lemma*} yields 
	\[
	\begin{aligned}
		\quad ~ \|y^{k+1}-y^*(x^{k+1})\|^2 \leq   
		\left(1- \frac{1}{2}{\mu }\beta\right) \|y^{k}- y^*(x^{k}) \|^2 + \left(1+ \frac{2}{\mu \beta} \right) \frac{L_{f_{y}}^2}{\mu^2}\|x^{k+1}-x^{k}\|^2.
	\end{aligned}
	\]
	Similarly, for any $\varepsilon > 0$,
	\[
	\begin{aligned}
		\|v^{k+1}_i-v^*_{i}(x^{k+1})\|^2
		\leq &
		(1+ \varepsilon) \|v_i^{k+1}- v_i^*(x^{k}) \|^2 + (1+ \frac{1}{\varepsilon }) \|v_i^*(x^{k+1}) -v_i^*(x^k)\|^2.
	\end{aligned}
	\]
	Taking $\varepsilon = \frac{1}{2}\mu \eta_k$,  Lemma \ref{lemma*} and \ref{lemmayv} yield that
	\[
	\begin{aligned}
		\|v_i^{k+1}-v^*_i(x^{k+1})\|^2
		\leq & \left(1- \frac{1}{2} \mu \eta_k \right) \|v_i^{k}- v_i^*(x^{k}) \|^2 + 2\eta_k \left(1 + \frac{1}{2}\mu \eta_k\right) \frac{(L_{f_{yy}}C_{F_y} +L_{F_y}\mu)^2}{\mu^3}\|y^{k}-y^*(x^k)\|^2  \\
		&  + \left( 1 + \frac{2}{\mu \eta_k} \right)\frac{L_{v}^2}{\mu^6}	\|x^k-x^{k+1}\|^2.
	\end{aligned}
	\]
\end{proof}

We can show in the following result that the iterates generated by gMOBA admit a decreased property in terms of the intrinsic Lyapunov function values $V_i^k$ for each $i = 1, \ldots, m$.
\begin{proposition}\label{prop1}
	Suppose Assumptions \ref{Assump} holds, and $\beta\leq\frac{2}{\mu+L_{f_y}}$, $\eta_k\leq \frac{1}{L_{f_{y}}}$ for all $k$. Then the sequence of $x^k, y^k, v^k_i$ generated by gMOBA satisfies that for each $i = 1, \ldots, m$,
	\begin{equation}\label{eq10}
		V_i^{k+1} - V_i^{k} 
		\le 
		- \hat{\alpha}_k \|x^{k+1} - x^{k}\|^2
		- \hat{\beta}_k \|y^{k}- y^*(x^{k}) \|^2 - \hat{\gamma}_k  \|v_i^{k} - v_i^*(x^{k})\|^2,
	\end{equation}
	where $\hat{\alpha}_k, \hat{\beta}_k, \hat{\eta}_k$ are constants.
	Specially, there exist $c_\alpha, c_\eta > 0$ such that when $\alpha_k = \alpha \le \min\{c_ \alpha, \frac{\mu}{4L_{f_{y}}^2}\eta_k, \frac{\mu^7}{16L_v^2}\eta_k\}$ and $\eta_{k} = \eta \le c_\eta$, it holds
	\[
	\hat{\alpha}_k \ge \hat{\alpha}, \qquad \hat{\beta}_k \ge \hat{\beta}, \qquad \hat{\eta}_k \ge \hat{\eta},
	\]
	for some positive constants $\hat{\alpha}$, $\hat{\beta}$ and $\hat{\eta}$.
\end{proposition}

\begin{proof}
	Here we take 
	\[
	\begin{aligned}
		\hat{\alpha}_k &:=  \frac{1}{4\alpha_k}-\frac{L_\varphi}{2\mu^3} - \left(1+ \frac{2}{\mu \beta} \right) \frac{L_{f_{y}}^2}{\mu^2} -  \left( 1 + \frac{2}{\mu \eta_k} \right)\frac{L_{v}^2}{\mu^6}  , \\
		\hat{\beta}_k &:=  \frac{1}{2}{\mu }\beta - 2\alpha_k\left( L_{F_{x}} + \frac{C_{F_{\y}} L_{f_{xy}}}{\mu}\right)^2 - 2\eta_k \left(1 + \frac{1}{2}\mu \eta_k\right) 	\frac{(L_{f_{yy}}C_{F_y} +L_{F_y}\mu)^2}{\mu^3}, \\
		\hat{\eta}_k &:= \frac{1}{2} \mu \eta_k  -  2\alpha_k L_{f_{y}}^2\left(1-\eta_k \mu\right).
	\end{aligned}
	\]
	By Lemma \ref{lemmaFk}, we have that for each $i = 1, \ldots, m$,
	\[
	\begin{aligned}
		V_i^{k+1} - V_i^{k} 
		= \, & \Phi_i(x^{k+1})  + \| y^{k+1}- y^*(x^{k+1})\|^2  + \| v_i^{k+1}- v^*_i(x^{k+1})\|^2 \\
		&- \Phi_{i}(x^k) - \| y^{k}- y^*(x^{k})\|^2 - \| v_i^{k}- v^*_i(x^{k})\|^2 \\
		\le \, &-\frac{1}{2}\left(\frac{1}{2\alpha_k}-\frac{L_\varphi}{\mu^3}\right)\|x^{k+1}-x^k\|^2
		\\
		& + \| y^{k+1}- y^*(x^{k+1})\|^2 + 2\alpha_k\left( L_{F_{x}} + \frac{C_{F_{\y}} L_{f_{xy}}}{\mu}\right)^2 \|y^{k}-y^*(x^k)\|^2  - \| y^{k}- y^*(x^{k})\|^2 \\
		& + \| v_i^{k+1}- v^*_i(x^{k+1})\|^2 + 2\alpha_k L_{f_{y}}^2 \|v^{k}-v^*_i(x^k)\|^2- \| v_i^{k}- v^*_i(x^{k})\|^2.
	\end{aligned}
	\]
	Then we can get the conclusion by applying Lemma \ref{lemmayv} and \ref{lemmayv2}.
\end{proof}

\subsubsection{Convergence Towards Pareto Stationarity}
Next, before we give the proof of the convergence of gMOBA towards the Pareto stationary point, we recall the following equivalent characterization for the Pareto stationary points of $\Phi$, which is established in Lemma 3.2 of  \cite{tanabe2019proximal}.

\begin{lemma}\label{lem7} \cite{tanabe2019proximal}
	Let $d_{\ell}(x) $ be defined as the solution of following optimization problem:
	\begin{equation}\label{define_dl}
		\underset{x+d \in \mathcal{X}}{\mathrm{min}} \,\max_{i = 1, \ldots, m}\left\{ \nabla \varphi_i(x)^T d + g_i(x+d) - g_i(x) \right\}  + \frac{\ell}{2}\|d\|^2,
	\end{equation}
	where $\ell$ is a positive constant, then $\bar{x}$ is a Pareto stationary point of $\Phi$ if and only if $d_{\ell}(\bar{x}) = 0$.
\end{lemma}
Combining the above equivalent characterization for the Pareto stationary with the sufficient decrease property established in Proposition \ref{prop1}, we are ready to establish the main convergence property of our proposed gMOBA, which is new even for the BLO case. 
\begin{theorem}\label{mainthm}
	Suppose Assumptions \ref{Assump}  holds. There exist $c_\alpha, c_\eta > 0$ such that when step sizes $\alpha_k$, $\beta$ and $\eta_k$ are chosen as constants that satisfy
	$\alpha_k = \alpha \le \min\{c_ \alpha, \frac{\mu}{4L_{f_{y}}^2}\eta_k, \frac{\mu^7}{16L_v^2}\eta_k\}$,  $\beta\leq\frac{2}{\mu+L_{f_y}}$, and $\eta_{k} = \eta \le c_\eta$, any limit point of sequence $\{x^k\}$ generated by gMOBA is Pareto stationary of problem \eqref{MOBLO}.
\end{theorem}
\begin{proof}
	Telescoping inequalities \eqref{eq10} for $k = 0, \ldots, K$ yields that for each $i = 1, \ldots, m$,
	\begin{equation*}
			\sum_{k=0}^K \Big( \hat{\alpha} \|x^{k+1} - x^{k}\|^2 
			+ \hat{\beta} \|y^{k}- y^*(x^{k}) \|^2 
			+ \hat{\gamma}  \|v_i^{k} - v_i^*(x^{k})\|^2 \Big) \le V_i^0 - V_i^{K+1}.
	\end{equation*}
	By Assumption \ref{Assump}, $F_i(x^k,y^*(x^k))$ is bounded below for each $i$ and $k$, taking $K \rightarrow \infty$ in above inequalities implies that for each $i = 1, \ldots, m$. Thus
	\begin{equation*}
			\sum_{k=0}^\infty \Big( \hat{\alpha} \|x^{k+1} - x^{k}\|^2 
			+ \hat{\beta} \|y^{k}- y^*(x^{k}) \|^2
			+ \hat{\gamma}  \|v_i^{k} - v_i^*(x^{k})\|^2 \Big) < \infty,
	\end{equation*}
	which implies 
	$
	\|x^{k+1} - x^{k}\| \rightarrow 0$, $\|y^{k}- y^*(x^{k}) \| \rightarrow 0$, $ \|v_i^{k} - v_i^*(x^{k})\| \rightarrow 0$, $ i = 1, \ldots, m.$
	By the update of $x^{k+1} $, 
	according to the first-order optimality condition (see, e.g., Theorem 4.14 in  \cite{mordukhovich2013easy}) and Danskin's theorem (see, e.g., Theorem 2.55 in  \cite{mordukhovich2013easy}), there exists $\lambda^k \in \mathbb{R}^m$ such that $\lambda^k \ge 0$, $\sum_{i=1}^m \lambda^k_i = 1$, 
	\begin{equation}\label{eq:subopt}
		0 \in \sum_{i=1}^m \lambda^k_i \left(d_{\varphi_i}^k + \partial g_i(x^{k+1}) \right) + \frac{1}{\alpha}(x^{k+1} - x^k) + \mathcal{N}_{\mathcal{X}}(x^{k+1}),
	\end{equation}
	and $\lambda^k_i = 0$ for inactive $i$ for  $\max_{i=1,\dots,m}\{ \langle d_{\varphi_i}^k, x^{k+1} - x^{k}\rangle  + g_i(x^{k+1}) - g_i(x^k) \}
	$. As shown in the proof of Lemma \ref{lemmaFk}, we have for $i = 1, \ldots, m$,
	\begin{equation*}
		\begin{aligned}
			\| 	\nabla \varphi_{i}(x^k) - d_{\varphi_i}^k\| 
			\leq 
			\left( L_{F_{x}} + \frac{C_{F_{\y}} L_{f_{xy}}}{\mu}\right) \|y^{k}-y^*(x^k)\|
			+L_{f_{y}} \|v_i^{k}-v^*_i(x^k)\|.
		\end{aligned}
	\end{equation*}
	Then since $\|y^{k}- y^*(x^{k}) \| \rightarrow 0$, and $\|v_i^{k} - v_i^*(x^{k})\| \rightarrow 0$ for $i = 1, \ldots, m$, we have 
	\[
	\|\nabla \varphi_{i}(x^k) - d_{\varphi_i}^k\| \rightarrow 0, \quad i = 1, \ldots, m.
	\]
	Let $\bar{x}$ be any limit point of sequence $\{x^k\}$ and $\{x^l\}$ be the subsequence of $\{x^k\}$ such that $x^l \rightarrow \bar{x}$. Since $\|x^{k+1} - x^{k}\| \rightarrow 0$, we have $x^{l+1} \rightarrow \bar{x}$. By the continuity of $\nabla \varphi_{i}$, we have $d_{\varphi_i}^k \rightarrow \nabla \varphi_{i}(\bar{x}) $ for $i = 1, \ldots, m$. Since $\{\lambda^l\}$ are in a compact set, we can assume without loss of generality by taking a further subsequence that $\lambda^l \rightarrow \bar{\lambda}$ with $\bar{\lambda}$ satisfying $\sum_{i=1}^m \bar{\lambda}_i = 1$ and $\bar{\lambda} \ge 0$. Then by taking $k = l$ in \eqref{eq:subopt} and $l \rightarrow \infty$, because $g_i$ is locally Lipschitz continuous and $\partial g_i$, $\mathcal{N}_{\mathcal{X}}$ is outer semicontinuous, we have
	\[
	0 \in \sum_{i=1}^m \bar{\lambda}_i \left(\nabla \varphi_{i}(\bar{x}) + \partial g_i(\bar{x}) \right)  + \mathcal{N}_{\mathcal{X}}(\bar{x}).
	\]
	Then by Danskin's theorem and the first-order optimality sufficient condition for convex functions, we can obtain that $0$ is the minimizer of Problem \eqref{define_dl}
	and thus it follows from Lemma \ref{lem7} that $\bar{x}$ is Pareto stationary of $\Phi$.
\end{proof}

\section{L2O-gMOBA}
In many scenarios, the MOBLO problem \eqref{MOBLO} does not possess a unique Pareto optimal solution. To accurately represent the true Pareto front, it is necessary to identify as many Pareto optimal solutions as possible. This requires running our proposed gMOBA algorithm multiple times from different initial points, which can be computationally expensive. Motivated by the recent promising performance of L2O, we introduce a beneficial L2O neural network called L2O-gMOBA in this section. L2O-gMOBA will be implemented as the initialization phase of our gMOBA algorithm, aiming to accelerate its convergence.

Instead of employing a general-purpose neural network, we draw inspiration from the unrolling technique used in the design of model-based L2O. L2O-gMOBA is constructed based on the unrolling concept applied to our proposed gMOBA algorithm. In this study, we specifically focus on the case where $g(x) = 0$ and $\mathcal{X} = \mathbb{R}^{n}$. However, it is important to note that L2O-gMOBA can be extended to handle general cases involving arbitrary $g(x)$ and $\mathcal{X}$.

\subsection{L2O-gMOBA Neural Network}
We propose the construction of the L2O-gMOBA neural network by unrolling our gMOBA algorithm and truncating it to a fixed number of $K$ iterations. However, in Algorithm \ref{alg1}, the update rule for the variable $x$ is implicit. Specifically, we need to solve the optimization problem \eqref{updatex} at each iteration to update $x$. As a result, we cannot directly utilize the iterative scheme of gMOBA as a layer in the neural network.

We observe from \eqref{eq:subopt} that the updated variable $x^{k+1}$ generated by gMOBA at each iteration satisfies certain conditions. In fact, there exists $\lambda^k \in \mathbb{R}^m$ such that $\lambda^k \ge 0$, $\sum_{i=1}^m \lambda^k_i = 1$, and $x^{k+1} = x^k - \alpha\sum_{i=1}^m \lambda^k_i d_{\varphi_i}^k$. Leveraging this observation, we propose to treat $\lambda^k \in \mathbb{R}^m$ as the learnable parameters of the neural network, enabling us to construct the network using an explicit iterative scheme resembling gMOBA.
To further enhance the convergence speed of the neural network, we introduce an additional parameter $\gamma^k$ to control the overall step sizes of gMOBA at each iteration. This parameter becomes part of the neural network's learnable parameters.

In summary we define the parameter set $\Theta = \{ \{\lambda^k\}_{k=0}^K, \{\gamma^k\}_{k=0}^K \}$ as the parameters to be learned in the L2O-gMOBA neural network. The $k$-th layer of the L2O-gMOBA neural network can be constructed using the following explicit iterative scheme:
\begin{align*}
	y^{k+1} = & y^k - \gamma^k\beta \nabla_y f(x^k, y^k), \\
	v_i^{k+1} = & v_i^k - \gamma^k\eta (\nabla_{yy}^2 f(x^k, y^k) v_i^k - \sum_{i=1}^k \lambda_i^k \nabla_y F_i(x^k, y^k)) \quad \forall\, i = 1,\dots,m, \\
	x^{k+1} = & x^k - \gamma^k\alpha \sum_{i=1}^{m} \lambda_i^k (\nabla_x F_i(x^k, y^k) - \nabla_{xy}^2 f(x^k, y^k) v_i^k) ,
\end{align*}
with regarding the variables $y^k$, $v^k$ and $x^k$ 
as the units of the $k$-th hidden layer. Figure \ref{fig:neural} illustrates the structure of the $(k-1)$-th layer in the L2O-gMOBA neural network. As mentioned earlier, it is necessary for $\lambda^k$ to satisfy $\lambda^k \geq 0$ and $\sum_{i=1}^m \lambda_i^k = 1$. To ensure these conditions, we employ the softmax function $\sigma$ at each layer of the L2O-gMOBA neural network. The softmax function $\sigma$ guarantees that $\sum_{i=1}^{m} \sigma({\lambda}_i^k) = 1$ and $\sigma({\lambda}_i^k) \geq 0$ for any $\lambda^k \in \mathbb{R}^m$.

\begin{figure}
	\centering 
	\includegraphics[width=\textwidth]{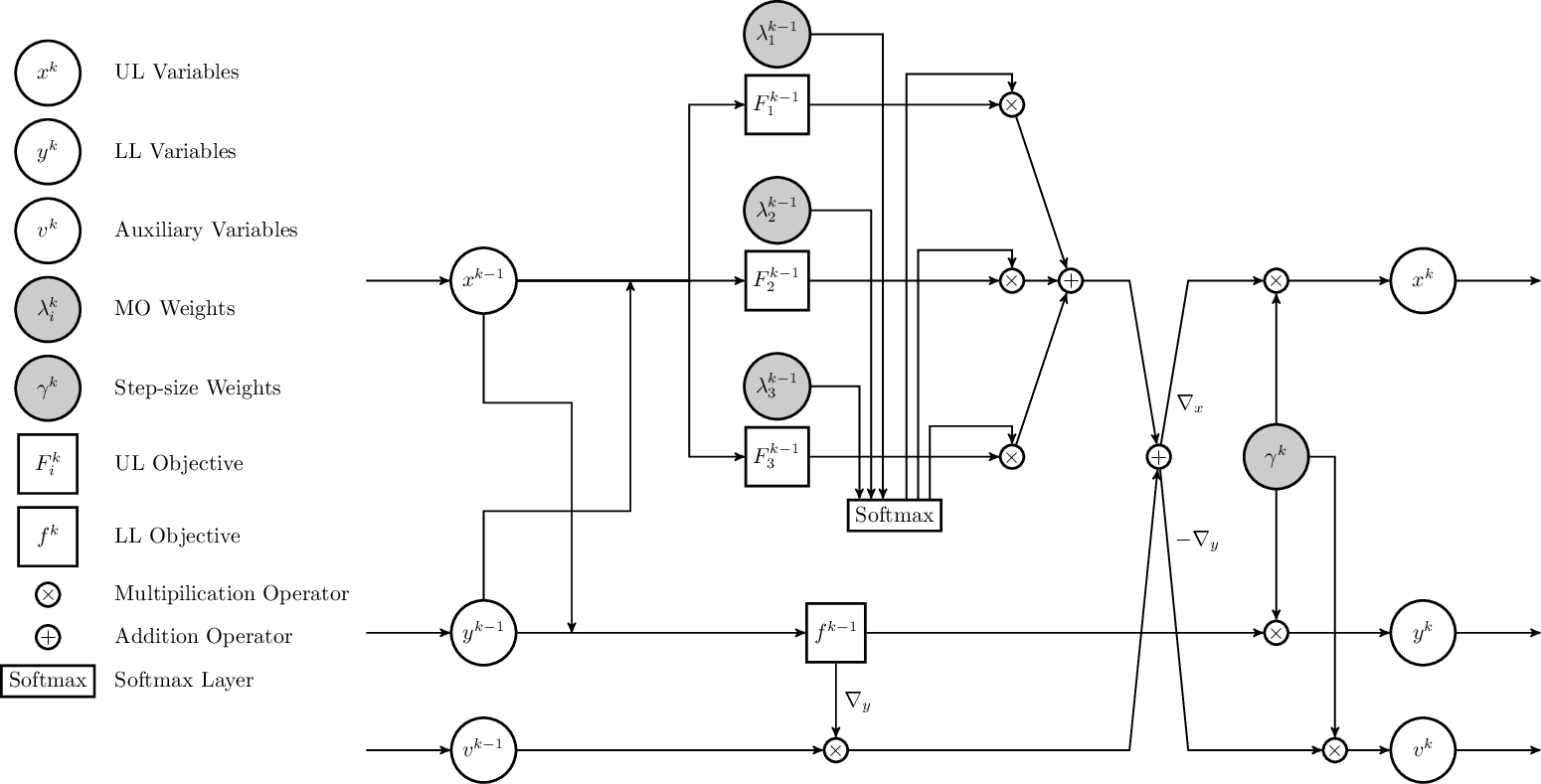}
	\caption{The architecture of the neural network for gMOBA}
	\label{fig:neural}
\end{figure}

\subsection{Training Procedure}
\label{sec:train}
The efficiency of L2O neural networks has been demonstrated in various applications, employing different training methods in the existing literature (e.g., see \cite{andrychowicz2016learning, yang2017admm}). Most of these studies train the neural network to achieve good performance across different optimization problems with identical structures and problem data from the same distribution.
However, in our work, the proposed L2O-gMOBA neural network is specifically designed to expedite the convergence of gMOBA for a fixed multi-objective optimization problem, considering different random initial points. Consequently, our L2O-gMOBA neural network is trained based on the given multi-objective optimization problem.

Following the training methodology presented in \cite{lin2020controllable}, we train the L2O-gMOBA neural network as follows. Initially, we set the step sizes of gMOBA and randomly initialize the network parameters $\Theta$, and train it with a gradient-based method like Adam \cite{kingma2014adam}.
For each training iteration, we sample a vector $p$ from a Dirichlet distribution $\mathrm{Dir}(\mathrm{1}_m)$ and randomly select an initial point $(x^0, y^0, v^0)$. Subsequently, we update the parameter $\Theta$ by minimizing the following loss function:
\[
\mathcal{L}_1(\Theta) := \mathbb{E}_{p\sim \mathrm{Dir}(\mathrm{1}_m)} \left[\sum_{i=1}^m p_i F_i(x^K(\Theta), y^K(\Theta))\right],
\]
where $(x^K(\Theta), y^K(\Theta), v^K(\Theta))$ represents the output of the L2O-gMOBA neural network at the final layer, given the initial point $(x^0, y^0, v^0)$ as input. The design of this loss function is inspired by \cite{lin2020controllable}. 
Additionally, we also evaluate the following alternative loss functions in our numerical experiments,
\begin{align*}
	\mathcal{L}_2(\Theta) := & \mathbb{E}_{p\sim \mathrm{Dir}(\mathrm{1}_m)} \left[\sum_{i=1}^m p_i F_i(x^K(\Theta), y^K(\Theta))\right] + f(x^K(\Theta), y^K(\Theta)), \\
	\mathcal{L}_3(\Theta) := & \max_{i=1, \dots, m}F_i(x^K(\Theta), y^K(\Theta)), \\
	\mathcal{L}_4(\Theta) := & \max_{i=1, \dots, m} F_i(x^K(\Theta), y^K(\Theta)) + f(x^K(\Theta), y^K(\Theta)).
\end{align*}
The loss function $\mathcal{L}_2$ incorporates lower-level objective function $f$ to enhance the lower-level optimality of the output. The loss function $\mathcal{L}_3$ and $\mathcal{L}_4$ are inspired by the Tchebycheff approach which is widely adopted in multi-objective optimization \cite{miettinen1999nonlinear}.

\section{Experiments}
In this section, we present experiments to illustrate the convergence property of our proposed gMOBA on the toy examples. All the experiments are coded in Python and implemented on 
a server with 64-bit Ubuntu 20.04.3 LTS, Intel(R) Xeon(R) Gold 5218R CPU @ 2.10GHz and 256.00 GB memory.
To illustrate the efficiency of our algorithms, we compare the proposed gMOBA and L2O-gMOBA with several algorithms including two gradient-based bilevel multi-objective algorithms MOML \cite{ye2021multi} and MORBiT \cite{gu2022min}, and three classical evolutionary single level multi-objective optimization algorithms NSGA-II \cite{deb2002fast}, NSGA-III \cite{jain2013evolutionary} and C-TAEA \cite{li2018two}. These evolutionary methods can provide fairly good results on small-size problems but fail when the dimension of the problem is large. 

Implementation of L2O-gMOBA is highly dependent on Pytorch \cite{pytorch}. Based on this, we define the L2O optimizer and train it through auto-grad technology, specifically, we train our L2O optimizer with Adam \cite{kingma2014adam} and use 0.01 as our learning rate. 
For each iteration in the training process, we sample the initial point $x^0$ from a $n$-dimension multivariate normal distribution $\mathcal{N}(0_n, I_{n\times n})$. 

\subsection{Numerical Problems}
To elaborate on how our proposed gMOBA converges to a Pareto stationary point, we consider the MOBLO problem \eqref{MOBLO} with
\[
F_i(x, y) =  \frac12 \begin{bmatrix}
	x\\ y
\end{bmatrix}^\mathup{T} A_i \begin{bmatrix}
	x\\ y
\end{bmatrix} + a_i^\mathup{T} \begin{bmatrix}
	x\\ y
\end{bmatrix},\quad g_i(x) = 0, 
\]
where $A_i \in \mathbb{R}^{(2n) \times (2n)}$, $a_i \in \mathbb{R}^{2n}$, $i = 1, \dots,m$, and
\[
f(x, y) = \frac12 y^\mathup{T} B y + x^\mathup{T} y + \frac{\mu}2 \|y\|^2,
\]
where $B \in \mathbb{R}^{n \times n}$ and $\mu>0$. In the following, each $A_i$ and $B$ are positive semidefinite. Recall that the {\it Pareto front} is the collection of vector objective values for all Pareto optimal solutions. Note the lower-level solution is given by 
\[
	y^*(x) = -(B^\mathup{T} B + \mu I)^{-1} x.
\] 
For this kind of toy example, the entire Pareto front can be easily obtained once all the problem data are determined by numerical calculation when $m$ is small, which helps us to evaluate the convergence property of the solvers. In all the following experiments, each component of $a_i$ is generated randomly from the uniform distribution on $[-1, 1]$, and each $A_i$ and $B$ are randomly generated such that their eigenvalues are between $0$ and $1$. 

To compare the performance of our proposed method, gMOBA, with classical evolutionary algorithms designed for single-level multi-objective optimization, we implemented these algorithms on a single-level multi-objective optimization reformulation where the lower-level problem is replaced by its optimality condition:
\begin{equation}\label{numerical2}
	\begin{aligned}
		\mathop{\min}\limits_{x \in \mathbb{R}^{n},\, y \in \mathbb{R}^{n}} 
		& (F_1(x,y), \ldots, F_m(x,y)), \\
		\mathrm{s.t.} \quad & 
		(B^\mathup{T} B y + \mu I) y + x = 0.
	\end{aligned}
\end{equation}
And the corresponding relaxed problem is 
\begin{equation}\label{numerical3}
	\begin{aligned}
		\mathop{\min}\limits_{x \in \mathbb{R}^{n},\, y \in \mathbb{R}^{n}} 
		& (F_1(x,y), \ldots, F_m(x,y)), \\
		\mathrm{s.t.} \quad & 
		(B^\mathup{T} B y + \mu I) y + x \le \varepsilon, \quad - (B^\mathup{T} B y + \mu I) y - x \le \varepsilon,
	\end{aligned}
\end{equation}
where $\varepsilon$ is a small positive number. Even for such a simplified problem, evolutionary algorithms struggle to find the Pareto front when the dimension of the problem is large, which is shown in the following experiments. 

\subsection{Metrics}
Since the MOBLO problem involves both multi-objective optimization and bilevel optimization, we need to evaluate a solution from both sides, thus, the following metrics are introduced in our experiments. 

To capture the performance from the multi-objective optimization perspective, we will use five metrics, purity \cite{bandyopadhyay2004multiobjective}, generational distance (GD) \cite{van1999multiobjective}, spread ($\Gamma$, $\Delta$) \cite{custodio2011direct} and spacing (SP) \cite{schott1995fault} to compare the performance of different solvers. The purity metric is computed by the number $|Y_N \cap Y_P|/|Y_N|$, where $Y_N$ is the Pareto front approximation obtained by the solver and $Y_P$ is a discrete representation of the real Pareto front. The GD metric is given by the formula $\sqrt{\sum_{y_1 \in Y_N} \min_{y_2 \in Y_P} \|y_1 - y_2\|^2 }/ |Y_N|$. Both the purity and GD metrics assess the ability of a solver to obtain points that are Pareto optimal.
The spread ($\Gamma$) metric is computed with $\max_{j \in \{ 1,\dots,m\} } (\max_{i \in \{0, \ldots, N\}}\delta_{i,j} )$  while the spread ($\Delta$) metric is computed with $\max_{j \in \{1,2,\dots, m \}} \left\{ 
(\delta_{0, j} + \delta_{N, j} + \sum_{i=1}^{N-1} |\delta_{i, j} - \bar{\delta}_j|) / (\delta_{0, j} + \delta_{N, j} + (N - 1) \bar{\delta}_j)
\right\}$, where $\delta_{i,j} = (F_{i+1,j}-F_{i,j})$ and the objective function values have been sorted by increasing order for each objective $F_j$, $\bar{\delta}_j$ for $j = 1,2,\dots, m$ is the mean of all distance $\delta_{i, j}$ for $i = 1,2,\dots, N-1$.
The SP metric is computed by $\sqrt{\sum_{y_1 \in Y_N}  (\bar{d} - d^1(y_1, Y_N \backslash\{y_1\}) )^2/(|Y_N|-1)}$, where $d^1(y_1, Y_N \backslash\{y_1\})$ denotes the $l_1$ distance of $y_1$ to the set $Y_N \backslash\{y_1\}$, i.e., the minimal $l_1$ distance of $y_1$ to the rest points in $Y_N$, and $\bar{d}$ is the mean of $d^1(y_1, Y_N \backslash\{y_1\})$ for $y_1 \in Y_N$. Both the spread and SP metrics capture the extent of the spread achieved in the Pareto front approximation obtained by the solver.

From the bilevel optimization perspective, we consider the optimality and the feasibility of the solution. Since the Pareto Front can be easily obtained, we can calculate the distance between terminated points to the Pareto Front to evaluate its optimality. We denote such a distance as $d_p$ and calculate it by $d_p = d((x^k, y^*(x^k)), \mathrm{PF}) / n$ where $\mathrm{PF}$ denotes the set of Pareto Front, the $n$ in the denominator serves for normalization. Due to the hierarchical structure of MOBLO, its feasibility means the optimality of lower-level problem, i.e., $\mathrm{Feasibility} = f(x^k, y^k) - f(x^k, y^*(x^k))$.

\subsection{Numerical Results}

In the first experiment, we choose $n = 100$, $m = 2$, and set $\mu = 0.1$. 
Every entry of initial point $x^0$ is chosen randomly from a standard normal distribution $\mathcal{N}(0, 1)$. 
For gMOBA, the parameters of stepsize are chosen as $\alpha_k = 0.0025, \beta = 1.0, \eta = 0.1$. For MOML in \cite{ye2021multi}, the parameters of stepsize are chosen as $\nu = 0.01, \mu = 1$, and we solve the lower level problem with 5 steps. For MORBiT in \cite{gu2022min}, the parameters of stepsize are chosen as $\alpha = 0.01, \beta = 1, \gamma = 0.003$.
All the solvers are terminated when the maximum change among multi-objective values is less than $10^{-4}$. We generate 100 initial points and run the algorithms from each of them. For the evolutionary methods, we set the maximum generation as $1000$ and the maximum number of evaluations as $100000$. Since they cannot be applied to the bilevel optimization problem directly, we apply them to the simplified problem \eqref{numerical2} where the lower-level problem is replaced by its optimality condition. However, when the dimension of the problem is large, the evolutionary methods fail to find feasible points. Therefore, we apply these methods to the relaxed problem \eqref{numerical3}, specifically, we set $\varepsilon = 10$ for C-TAEA, and $\varepsilon = 1$ for NSGA-II and NSGA-III, since smaller values of $\varepsilon$ lead to failure of the algorithms in our practice, i.e., the algorithms will find no feasible points with such small $\varepsilon$.  

\begin{figure}[ht]
	\centering
	\includegraphics[width = 0.5\textwidth]{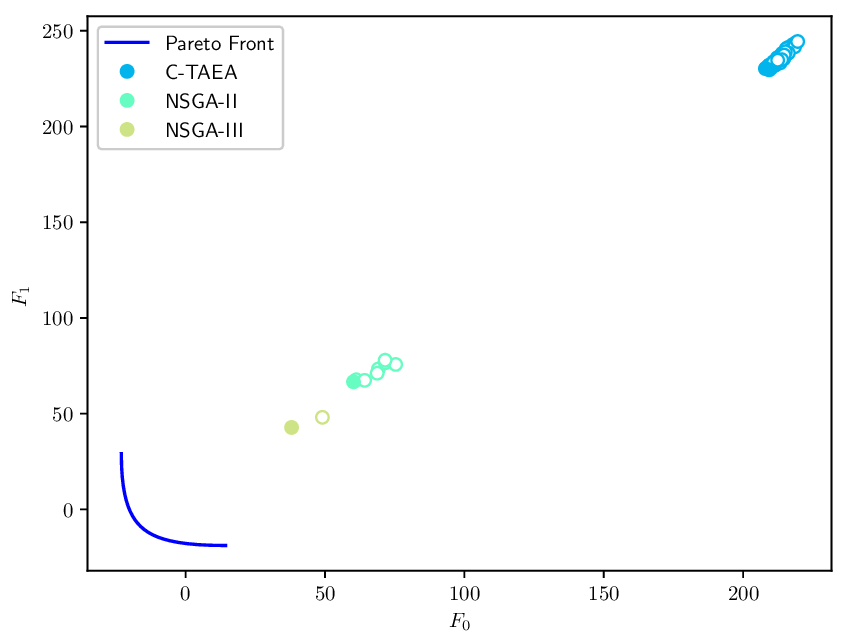} 
	\caption{Evolutionary methods on 2d toy example ($m = 2, n=100$) from 100 random initial points} \label{Fig:toy2d2}
\end{figure}

\begin{figure}[ht]
	\centering
	\includegraphics[width = .9\textwidth]{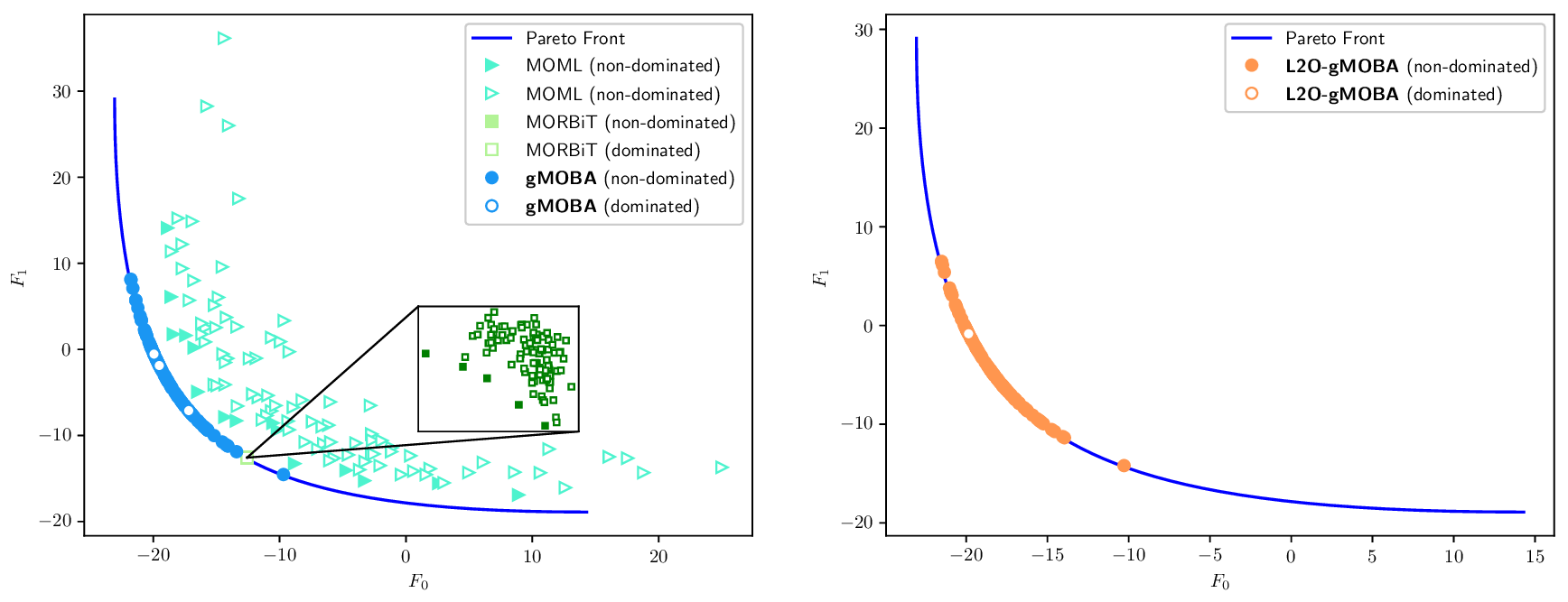} 
	\caption{Gradient-based methods on 2d toy example ($m = 2, n=100$) from 100 random initial points} \label{Fig:toy2d1}
\end{figure}

The evolutionary methods are shown in Figure \ref{Fig:toy2d2}, while the gradient-based results are presented in Figure \ref{Fig:toy2d1}. 
It can be seen that iterates generated by gradient-based methods can get close to the Pareto front well when the evolutionary methods fail to find the Pareto front under this setting. Therefore, we will only focus on the numerical performance of gradient-based methods in the following experiments. Besides, from the graph, we can see that gMOBA can converge to different parts on the Pareto front from different initial points, while MORBiT leads to the same point on the Pareto front from different initial points, and MOML fails to reach the Pareto front close enough.
However, it can be seen that MOML fails to converge to the Pareto Front from most of the initial points.

Next, we generate $5$ different problems with different randomly generated matrices $A_i$ and $B$ and vector $a_i$ and conduct experiments to compare the convergence properties of our proposed gMOBA and MOML under the same setting as above. 

The results of the experiments are shown in Tables \ref{table:toy1}--\ref{table:toy2}. All the results are reported with the mean value and the standard deviation of the 5 repetitions. All the methods are terminated either the change of multi-objective values is less than $10^{-4}$ or the distance from iterate to the Pareto front is less than $0.05$. 
As the updates of $(x^k, y^k, v^k_i)$ in gMOBA are parallelizable, we record the largest computation time for updating $x^k, y^k$ and $v^k_i$ as the running time of gMOBA at each iteration. Running times of MOML reported in the table are also the parallelized time.
Though gMOBA may be a little slower than MOML, it outperforms MOML in almost all the other metrics. 
The purity metric also shows that gMOBA can converge to the true Pareto front quickly in most cases, while MOML and MORBiT can only find part of the Pareto front. From the table, we 
observe that although L2O-gMOBA requires significant training time, it offers an attractive efficiency improvement when we need to find a large amount of Pareto points. This scenario often arises in multi-objective optimization, making this improvement highly beneficial. 

\begin{table}[th]\footnotesize
	\centering 
	\caption{MOBLO metrics for toy example ($m = 2$, $n = 100$)}
	\label{table:toy1}
	\begin{tabular}{c ccccc}
		\toprule 
		Method & Training Time (s) & Time (s) & purity (\%) & $d_p$ & Feasibility 
		\\
		\midrule
		MOML   & - & 2.2 $\pm$ 0.0 &  19.2 $\pm$ 4.4 &  3.86 $\pm$ 1.17 &  0.8 $\pm$ 0.3 \\
		MORBiT & - & 11.3 $\pm$ 4.8 &  36.4 $\pm$ 38.1 &  0.12 $\pm$ 0.02 &  8e-07 $\pm$ 1e-06\\
		\textbf{gMOBA}  & - & 0.9 $\pm$ 0.4 &  97.8 $\pm$ 2.3 &  0.10 $\pm$ 0.00 &  8e-03 $\pm$ 1e-02\\
		\textbf{L2O-gMOBA}  & 53.2$\pm$3.0 & 0.7 $\pm$ 0.4 &  98.6 $\pm$ 0.5 &  0.10 $\pm$ 0.00 &  8e-03 $\pm$ 1e-02\\
		\bottomrule
	\end{tabular}
\end{table}

\begin{table}[th]\footnotesize\tabcolsep 16pt
	\centering 
	\caption{MOBLO metrics for toy example ($m = 2$, $n = 100$) (cont'd)}
	\label{table:toy2}
		\begin{tabular}{c cccc}
			\toprule 
			Method  & GD & spread ($\Gamma$) & spread ($\Delta$) & SP 
			\\
			\midrule
		MOML & 1.30 $\pm$ 0.71 &  6.15 $\pm$ 1.23 &  0.85 $\pm$ 0.07 &  1.58 $\pm$ 0.62\\
		MORBiT & 0.91 $\pm$ 1.00 &  0.01 $\pm$ 0.01 &  0.84 $\pm$ 0.08 &  0.00 $\pm$ 0.00\\
		\textbf{gMOBA}  & 0.29 $\pm$ 0.30 &  2.72 $\pm$ 0.66 &  0.87 $\pm$ 0.03 &  0.27 $\pm$ 0.15\\
		\textbf{L2O-gMOBA} & 0.29 $\pm$ 0.31 &  1.49 $\pm$ 1.09 &  0.84 $\pm$ 0.02 &  0.17 $\pm$ 0.15\\
			\bottomrule
	\end{tabular}
\end{table}

In the second experiment, we choose $m = 3$ and keep all the other settings the same as in the first experiment. The results are collected in Table \ref{table:toy3}. From the table, we can see that this problem is more difficult than the previous one since the purities of all the methods are much lower than the previous one. However, gMOBA and L2O-gMOBA still outperform the other methods. 
\begin{table}[th]\footnotesize 
	\centering 
	\caption{MOBLO metrics for three objectives toy example ($m = 3, n = 100$)}
	\label{table:toy3}
	\begin{tabular}{c ccccc}
			\toprule 
			Method & Training Time (s) & Time (s) & purity (\%) & $d_p$ & Feasibility 
			\\
			\midrule
		MOML   & - & 49.1 $\pm$ 5.2 &  6.8 $\pm$ 2.3 &  23.55 $\pm$ 38.04 &  46.1 $\pm$ 5.2
		\\
		MORBiT & - & 31.2 $\pm$ 15.4 &  5.4 $\pm$ 5.9 &  0.14 $\pm$ 0.04 &  6e-07 $\pm$ 1e-07 \\
		\textbf{gMOBA}  & - & 6.4 $\pm$ 1.4 &  25.0 $\pm$ 13.7 &  0.14 $\pm$ 0.01 &  5e-04 $\pm$ 6e-04\\
		\textbf{L2O-gMOBA}  & 57.3 $\pm$ 0.5 & 5.3 $\pm$ 1.5 &  14.4 $\pm$ 6.3 &  0.13 $\pm$ 0.01 &  1e-03 $\pm$ 1e-03\\
			\bottomrule
	\end{tabular}
\end{table}

The third experiment is about a higher dimension experiment for both two and three objectives setting, i.e., we choose $m = 2$ or $3$, and $n = 1000$. The results are collected in Table \ref{table:toy4}. From the table, we can see that the proposed gMOBA and L2O-gMOBA give high-quality solutions in a short time even when the dimension is high, which implies that our proposed methods scale well.

\begin{table}[th]\footnotesize\tabcolsep 16pt
	\centering 
	\caption{MOBLO metrics for high dimensional toy example ($n = 1000$)}
	\label{table:toy4}
	\resizebox{\textwidth}{!}{
	\begin{tabular}{c ccccc}
			\toprule 
			Method & Training Time (s) & Time (s) & purity (\%) & $d_p$ & Feasibility 
			\\
			\midrule
			\multicolumn{6}{c}{$m = 2$}\\
			\midrule
		MOML   & - & 2.8 & 30 & 52.4 & 13.84  \\
		MORBiT & - & 24.4 & 2 & 1.39 & 1e-7  \\
		\textbf{gMOBA}  & - & 4.2 & 93 & 0.15 & 1e-3 \\
		\textbf{L2O-gMOBA}  & 247.6 & 3.8 & 95 & 0.14 & 9e-6  \\
			\midrule
			\multicolumn{6}{c}{$m = 3$}\\
			\midrule
		MOML   & - & 79.4 & 1 & 1365.2 & 867.6  \\
		MORBiT & - & 66.2 & 1 & 1.57 & 4e-6  \\
		\textbf{gMOBA}  & - & 27.8 & 9 & 1.23 & 1e-4 \\
		\textbf{L2O-gMOBA}  & 234.2 & 26.7 & 12 & 1.17 & 9e-6  \\
			\bottomrule
	\end{tabular}
	}
\end{table}

We also investigated the efficacy of L2O-gMOBA with different loss functions on small size problems where $m = 2$, $n = 5$ and find this method showcased significant improvements across multiple performance metrics while the choice of losses we consider aforementioned do not exhibit much differences. Specifically, in the training stage, L2O-gMOBA with 4 choices of losses takes around 8.8 seconds. In the evaluated stage, L2O-gMOBAs with $K =100$ consistently yielded higher purity percentages (ranging from 82.1\% to 83.4\%) with less time (ranging from 0.94s to 0.97s), compared to gMOBA's 76.1\% purity with 1.25s. Moreover, the $d_p$ values and the performance on feasibility for L2O-gMOBA remained competitive compared to gMOBA. All the methods get results of 0.05 on $d_p$ values with feasibility less than $10^{-5}$. These findings highlight L2O-gMOBA's potential to enhance the efficiency of gMOBA while showcasing resilience to variations in loss functions, thus we just report the results based on the first loss in our numerical experiments.

\section{Conclusion}

In this paper, we presented a simple yet highly efficient gradient-based algorithm (called gMOBA) for MOBLO that is guaranteed to converge to Pareto stationary solutions, and empirically converge to diverse Pareto optimal solutions, compared to existing methods. To accelerate the convergence of gMOBA, we introduce a beneficial L2O neural network (called L2O-gMOBA) implemented as the initialization phase of our gMOBA algorithm. Comparative results of numerical experiments are presented to illustrate the efficiency of L2O-gMOBA.

\section*{Declarations}

\noindent\textbf{Funding:} Yang's work is supported by the Major Program of the National Natural Science Foundation of China (Grant No. 11991020 and Grant No. 11991024). Yao's work is supported by National Science Foundation of China (Grant No. 12371305). Zhang's work is supported by National Science Foundation of China (Grant No. 12222106), Guangdong Basic and Applied Basic Research Foundation (Grant No. 2022B1515020082) and Shenzhen Science and Technology Program (Grant No. RCYX20200714114700072).

\vspace{5pt}

\noindent\textbf{Conflict of Interest:} The authors declare that they have no conflict of interest.

\end{document}